\documentclass[11pt,a4paper,reqno]{amsart}
\usepackage{amsfonts,amsthm,amsmath,amscd,mathrsfs}
\usepackage[utf8]{inputenc}
\usepackage{t1enc}
\usepackage[mathscr]{eucal}
\usepackage{indentfirst}
\usepackage{graphicx,graphics,pict2e}
\usepackage{epic}
\usepackage{lipsum}
\usepackage{stmaryrd}
\usepackage{authblk}
\usepackage{subcaption}
\usepackage[symbol]{footmisc}
\numberwithin{equation}{section}
\usepackage[margin=2.9cm]{geometry}
\usepackage{epstopdf}
\RequirePackage{doi}
\usepackage{hyperref}
\allowdisplaybreaks
\usepackage{cite}

\usepackage{verbatim,amssymb,amsfonts}
\usepackage{mathrsfs}
\usepackage{mathtools}
\usepackage{tikz-cd}
\usepackage{indentfirst}
\usepackage{slashed}
\usepackage{thmtools}
\usepackage{setspace}
\usepackage{enumerate}
\usepackage{accents}

\theoremstyle{plain}
\newtheorem{theorem}{Theorem}[section]
\newtheorem{lemma}[theorem]{Lemma}
\newtheorem{corollary}[theorem]{Corollary}

 \theoremstyle{definition}
\newtheorem{definition}[theorem]{Definition}
\newtheorem{remark}[theorem]{Remark}

\DeclarePairedDelimiterX{\inp}[2]{\langle}{\rangle}{#1, #2}

\newcommand{\cE}{{\mathcal E}}

\newcommand{\ba}{\begin{eqnarray}}
\newcommand{\na}{\end{eqnarray}}
\newcommand{\ban}{\begin{eqnarray*}}
\newcommand{\nan}{\end{eqnarray*}}

\newcommand{\N}{{\mathbb N}}

\renewcommand{\thefootnote}{\fnsymbol{footnote}}

\makeatletter
\g@addto@macro{\endabstract}{\@setabstract}
\newcommand{\authorfootnotes}{\renewcommand\thefootnote{\@fnsymbol\c@footnote}}%
\makeatother

\makeatletter
\@namedef{subjclassname@2020}{%
  \textup{2020} Mathematics Subject Classification}
\makeatother

\title[]{Line graphs and Nordhaus-Gaddum-type \\ bounds for self-loop graphs}

\date{\today}

\begin{document}

\begin{center}
    \vspace{-1cm}
	\maketitle
	
	\normalsize
    \authorfootnotes
    Saieed Akbari\textsuperscript{1},
    Irena M. Jovanovi\'c\textsuperscript{2},
    Johnny Lim\footnote[1]{Corresponding author.}\textsuperscript{3}
	\par \bigskip

        \textsuperscript{1}
        \small{Department of Mathematical Sciences, Sharif University of Technology, Tehran, Iran}

        \textsuperscript{2}
        \small{School of Computing, Union University, Belgrade, Serbia}

	\textsuperscript{3}
        \small{School of Mathematical Sciences, Universiti Sains Malaysia, Penang, Malaysia}\par \bigskip
	
\end{center}

\address{Department of Mathematical Sciences, Sharif University of Technology, Tehran, Iran
}
\email{s\_akbari@sharif.edu}

\address{School of Computing, Union University, Belgrade, Serbia
}
\email{irenaire@gmail.com}

\address{School of Mathematical Sciences, Universiti Sains Malaysia, Malaysia
}
\email{johnny.lim@usm.my}

\vspace{-0.5cm}

\begin{abstract}
Let $G_S$ be the graph obtained by attaching a self-loop at every vertex in $S \subseteq V(G)$ of a simple graph $G$ of order $n.$ In this paper, we explore several new results related to the line graph $L(G_S)$ of $G_S.$ Particularly, we show that every eigenvalue of $L(G_S)$ must be at least $-2,$ and relate the characteristic polynomial of the line graph $L(G)$ of $G$ with the characteristic polynomial of the line graph $L(\widehat{G})$ of a self-loop graph $\widehat{G}$, which is obtained by attaching a self-loop at each vertex of $G$. Then, we provide some new bounds for the eigenvalues and energy of $G_S.$ As one of the consequences, we obtain that the energy of a connected regular complete multipartite graph is not greater than the energy of the corresponding self-loop graph. 
Lastly, we establish a lower bound of the spectral radius in terms of the first Zagreb index $M_1(G)$ and the minimum degree $\delta(G),$ as well as proving two Nordhaus-Gaddum-type bounds for the spectral radius and the energy of $G_S,$ respectively. 

\vspace{1em}


\end{abstract}


\section{Introduction}
\label{intro}

Let $G=\left(V(G), E(G)\right)$ be a finite simple graph, where $V(G)$ is the set of vertices of $G$, and $E(G)$ is the set of its edges. If $|V(G)|=n$ and $|E(G)|=m$, we say that $G$ is a graph of order $n$ and size $m$. Let $V(G)=\{v_1, v_2,\ldots, v_n\}$ and $E(G)=\{e_1, e_2,\ldots, e_m\}$. The \emph{incidence matrix} of $G$ is the matrix $B(G)=(b_{ij})_{n \times m}$ whose rows and columns are indexed by the vertices and edges of $G$, respectively. The $(i,j)$-th entry $b_{ij}$ of $B(G)$ is equal to 1, if $v_i$ is incident with $e_j$, and $b_{ij}=0$, if $v_i$ and $e_j$ are not incident. The degree of the vertex $v_i$ will be denoted by $d_G(v_i)$, for $1\leq i\leq n$, while $\Delta(G)=\max\limits_{1\leq i\leq n} \{d_G(v_i)\}$ and $\delta(G)=\min\limits_{1\leq i\leq n} \{d_G(v_i)\}$ will be the maximum and the minimum degree of $G$, respectively. 
When there is no confusion, we write $\Delta$ and $\delta.$
In addition, if $d_G(v_i)=r$, for each $1\leq i\leq n$, $G$ is called an $r$-\emph{regular graph}. If either $d_G(v_i)=r$ or $d_G(v_i)=k$, for each $1\leq i\leq n$, then $G$ is an $(r,k)$-\emph{bidegreed graph}. If $G$ is a graph that is bipartite and bidegreed, then $G$ is an $(r,k)$-\emph{semiregular graph}.

Let $A(G)=(a_{ij})_{n \times n}$ be the \emph{adjacency matrix} of $G$, whose $(i,j)$-th entry $a_{ij}$ is defined by
$a_{ij}=1$, if $v_i$ and $v_j$ are adjacent vertices, and $a_{ij}=0$, otherwise. 
The \emph{characteristic polynomial} $P_G(x)=\det\left(xI_n-A(G)\right)$ of $G$ is the characteristic polynomial of its adjacency matrix, where $I_n$ is the $n \times n$ identity matrix. The \emph{(adjacency) eigenvalues} $\lambda_1(G)\geq\lambda_2(G)\geq\cdots\geq\lambda_n(G)$ of $G$ are the eigenvalues of $A(G)$. 
Since $A(G)$ is real and symmetric, the eigenvalues of $G$ are all real. If $\lambda_1>\lambda_2>\cdots>\lambda_t$ are the distinct eigenvalues of $G$, then the \emph{(adjacency) spectrum} of $G$ will be denoted by ${\rm Spec}(G)=\left( \begin{array}{cccc}
 \lambda_1 & \lambda_2 & \cdots & \lambda_t \\
a_1 & a_2 & \cdots & a_t \\
\end{array}
 \right),
$ where $a_i$, for $1\leq i\leq t$, is the algebraic multiplicity of the eigenvalue $\lambda_i$. In particular, $\lambda_1(G)$, as the largest eigenvalue of $G$, is called the \emph{spectral radius} (or \emph{index}) of $G$, and when there is no confusion of what $G$ is, we shall only write 
$\lambda_1$. It is well-known (see, for example, \cite{cvetkovic2010intro}) that $\sum\limits_{i=1}^{n} \lambda_i(G)=0$ and $\sum\limits_{i=1}^{n} \lambda_i^2(G)=2m$. Another frequently used results related to the eigenvalues of a graph are the \textit{Courant-Weyl inequalities} and the \textit{Interlacing Theorem}:

\begin{theorem} \cite[Theorem 1.3.15]{cvetkovic2010intro}\label{Weylthm}
Let $A$ and $B$ be $n\times n$ Hermitian matrices. Then,
\begin{equation}\label{eq1}\lambda_i(A+B)\leq \lambda_j(A)+\lambda_{i-j+1}(B), \, 1\leq j\leq i\leq n, \end{equation}
\begin{equation}\label{eq2}\lambda_i(A+B)\geq \lambda_j(A)+\lambda_{i-j+n}(B), \, 1\leq i\leq j\leq n. \end{equation}
\end{theorem}

\begin{theorem}\cite[Corollary 1.3.12]{cvetkovic2010intro}\label{interlacing}
Let $G$ be a graph with $n$ vertices and eigenvalues $\lambda_1 \geq \lambda_2 \geq \cdots \geq \lambda_n,$ and let $H$ be an induced subgraph of $G$ with $m$ vertices. If the eigenvalues of $H$ are $\mu_1 \geq \mu_2 \geq \cdots \geq \mu_m,$ then for $i=1,\ldots,m,$
\[
\lambda_{n-m+i} \leq \mu_i \leq \lambda_i.
\]    
\end{theorem}

The \emph{energy} $\mathcal{E}(G)$ of $G$ was first introduced by Gutman \cite{gutman1979} to be the sum of all absolute eigenvalues of $G$:
\begin{equation}
\label{eq:energy1}
\mathcal{E}(G)=\sum\limits_{i=1}^n |\lambda_i(G)|.
\end{equation}
Despite a lot of research have been done on studying different aspects of this graph invariant, graph energy still remains intriguing to researchers. For more details about graph energy and its application in mathematical chemistry, we refer the readers to, for example, \cite{gutman&furtula2019,gutmanli2016,gutman&ramane2019,gutmanlishi2012} and references therein.

Let $S\subseteq V(G)$ and $|S|=\sigma$, where $0\leq\sigma\leq n$. For $G$ a simple graph of order $n$ and size $m$, denote by $G_S$ the \emph{self-loop graph} of $G$ at $S$, i.e. $G_S$ is the graph of order $n$ and size $m$, obtained from $G$ by attaching a self-loop (or simply a loop) at each vertex from the set $S.$ For clarity, $m$ always denotes the number of \textit{ordinary edges} and the number of loops $\sigma$ is not incorporated in $m.$ 
When $\sigma=0$, we write $G$ instead of $G_S$, while when $\sigma=n$, we will use the notation $\widehat{G}$. 
The adjacency matrix $A(G_S)$ of $G_S$ takes the form $A(G_S) = A(G)+\mathfrak{I}_S$, where $\mathfrak{I}_S$ is the ``almost'' identity matrix, with exactly $\sigma$ ones on the
main diagonal corresponding to $S$ and all other entries equal to zero. 
The (adjacency) eigenvalues $\lambda_1(G_S)\geq\lambda_2(G_S)\geq\cdots\geq\lambda_n(G_S)$ of $G_S$ are the eigenvalues of the matrix $A(G_S)$, and since $A(G_S)$ is square and symmetric, these eigenvalues  are reals. 

Recently, Gutman \textit{et al.} \cite{gutman2022energy} has initiated the study of the spectral properties of $A(G_S)$, amongst which we have the following relations.

\begin{lemma}
\label{lem1}
\cite{gutman2022energy}
Let $S\subseteq V(G)$ and $|S|=\sigma.$ Let $G_S$ be a self-loop graph of order $n$ and size $m.$ Let $\lambda_1(G_S) \geq \cdots \geq \lambda_n(G_S)$ be the eigenvalues of $G_S.$ Then,
\begin{enumerate}[(i)]
\item $\displaystyle \sum^n_{i=1} \lambda_i(G_S)=\sigma,$
\item $\displaystyle \sum^n_{i=1} \lambda^2_i(G_S)=2m+\sigma.$
\end{enumerate}
\end{lemma}
The energy $\mathcal{E}(G_S)$ of $G_S$ \cite{gutman2022energy} of order $n$ and  $|S|=\sigma$ is defined as 
\begin{equation}
\label{eq:energy2}  
\mathcal{E}(G_S)=\sum\limits_{i=1}^{n} \left|\lambda_i(G_S)-\frac{\sigma}{n}\right|.
\end{equation}

Since the research topics related to the energy of self-loop graphs are relatively new, we refer the readers to the several papers that have been published on this subject so far: \cite{akbari2023selfloop,anchan20232,anchan2023,jovanovic2023,popat&shingala2023,ShettyBhat2023, SomborQuanti2023}. Application-wise, self-loop graphs have also classically been found to manifest in mathematical chemistry, cf. \cite{gutman1979topological,
gutman1990topological,mallion1974graphical}.

In the paper, we adapt commonly used notations in Spectral Graph Theory. The \emph{complement} $\overline{G}$ of a graph $G$ is the graph with the same set of vertices as $G$ such that two distinct vertices in $\overline{G}$ are adjacent whenever they are not adjacent in $G.$  The \emph{line graph} $L(G)$ of a graph $G$ is the graph whose vertices are the edges of $G$, with two vertices in $L(G)$ adjacent whenever the corresponding edges in $G$ have exactly one vertex in common. $K_n$ is the \emph{complete graph} of order $n$, while the \emph{complete multipartite graph} $K_{n_1, n_2,\ldots, n_k}$ is the complement of the graph $G=K_{n_1}\,\dot{\cup}\,K_{n_2}\,\dot{\cup}\cdots\dot{\cup}\,K_{n_k}$, where $\dot{\cup}$ stands for disjoint union. If $n_1=n_2=\cdots=n_k=p$, we will use the label $K_{k\times p}$. The \emph{first Zagreb index} $M_1(G)$ of $G$ is defined as \cite{gutman&trinajstic1972}: $M_1(G)=\sum\limits_{i=1}^{n} d^2_G(v_i)$, while the \emph{first Zagreb index} $M_1(G_S)$ of $G_S$ is \cite{ShettyBhat2023}: $M_1(G_S)=\sum\limits_{i=1}^{n} d^2_{G_S}(v_i)$, where $d_{G_S}(v_i)$ is the degree of the vertex $v_i$ in $G_S$. In particular, if $v \in S,$ then $d_{G_S}(v) = d_{G}(v) +2.$ 
In Section 4, we will also discuss semiregular graphs, for clarity, we explain the notion of \textit{semiregular matrices}, cf. \cite{Zhou2000spectral}. A nonnegative matrix $A$ is \textit{row-regular} (resp. \textit{column-regular}) if all of its row (resp. column) sums are equal. $A$ is \textit{row-semiregular} (resp. \textit{column-semiregular}) if there exists a permutation matrix $P$ such that $P^TAP=\begin{pmatrix} 0 & B \\ C & 0 \end{pmatrix}$ where $B$ and $C$ are both row-regular (resp. column-regular). Then, the matrix $A$ is said to be \textit{regular} (resp. \textit{semiregular}) if $A$ is both row-regular and column-regular (resp. row and column-semiregular). 
For the remaining basic terminology and additional details, the reader is referred to \cite{cvetkovic1995spectra} and \cite{cvetkovic2010intro}.

The paper is organized as follows. In Section~\ref{Sec2}, we establish that $\mathcal{E}(G\setminus S) < \mathcal{E}(G_S)$ for a set $S$ of independent vertices in $G,$ and prove several results related to the line graph $L(G_S)$ of $G_S$. Here, we show that every eigenvalue of $L(G_S)$ must be not less than $-2.$ The relation between the characteristic polynomials of the line graph $L(\widehat{G})$ of $\widehat{G}$ and its  counterpart $L(G)$ is derived. In Section~\ref{Sec3}, some bounds on the eigenvalues of a self-loop graph are given, together with some lower bounds on its energy. Moreover, we show that the energy of a connected regular complete multipartite graph is not greater than the energy of the corresponding self-loop graph. In Section~\ref{Sec4}, besides an upper bound, we give a lower bound for the spectral radius of a self-loop graph in terms of $M_1(G)$ and $\delta(G).$ We also present two Nordhaus-Gaddum-type bounds for the spectral radius of $G_S$. In Section~\ref{Sec5}, a Nordhaus-Gaddum-type bound for the energy of a self-loop graph in terms of its order and the number of loops is exposed.

\section{The line graph of a self-loop graph}
\label{Sec2}

Let $\mathcal{M}_{m,n}$ be the set of $m\times n$ complex matrices, and let $M\in\mathcal{M}_{m,n}$. We write $M^{\star}$ for the Hermitian adjoint of $M$. The \emph{singular values} $s_1(M) \geq s_2(M) \geq \cdots \geq s_n(M)$ of a matrix $M$ are the square roots of the eigenvalues of $MM^{\star}$. Since $A=A(G)$ is a real and symmetric square matrix, it holds that $s_i(A)=|\lambda_i(G)|$, for $1\leq i\leq n$, and hence the energy of $G$ of order $n$  is the sum of the singular values of its adjacency matrix \cite{nikiforov2007}.

In \cite{DaySo2008}, the following theorem regarding the singular values of a matrix has been proved:

\begin{theorem}\cite[Theorem 2.2]{DaySo2008} \label{Day&So}
For a partitioned matrix 
$C=\left(
\begin{array}{cc}
A & X \\
Y & B \\
\end{array}
\right)
$, where both $A$ and $B$ are square matrices, we have:
$$\sum\limits_{j} s_j(A)+\sum\limits_{j} s_j(B)\leq\sum\limits_{j} s_j(C).$$
\end{theorem}

Recall that vertices, or edges, of a graph are said to be \emph{independent} if they are pairwise non-adjacent.

\begin{theorem}\label{independent}
Let $G$ be a graph of order $n$ $(n\geq 2)$, and let $S$ be a set of independent vertices in $G$, such that $|S|=\sigma$, $1\leq\sigma\leq n-1$. Let $G_S$ be the graph obtained by attaching a loop at each vertex in the set $S$. Then, $$\mathcal{E}(G\setminus S) < \mathcal{E}(G_S),$$
where $G\setminus S$ is the graph obtained from $G$ by deleting all vertices from the set $S$.
\end{theorem}

\begin{proof}
The adjacency matrix $A(G_S)$ of $G_S$ is of the following form
$$A(G_S)=\left(
\begin{array}{cc}
A(G\setminus S) & M^{T} \\
M & I_{\sigma} \\
\end{array}
\right),
$$ 
where the matrix 
$M=(m_{ij})_{\sigma\times (n-\sigma)}$ satisfies that $m_{ij}=1$ if the vertex $i\in S$ is adjacent to the vertex $j\in V(G)\setminus S$, and $m_{ij}=0$, otherwise.

By Theorem \ref{Day&So}, we find:
\begin{align*}
&\mathcal{E}(G\setminus S)+\sigma = \sum\limits_{i=1}^{n-\sigma} |\lambda_i(A(G\setminus S))|+\sum\limits_{i=1}^{\sigma} |\lambda_i(I_{\sigma})| =  \sum\limits_{i=1}^{n-\sigma} s_i(A(G\setminus S))+ \sum\limits_{i=1}^{\sigma} s_i(I_{\sigma})\leq\\
& \sum\limits_{i=1}^{n} s_i(A(G_S))=\sum\limits_{i=1}^{n} |\lambda_i(A(G_S))|=\sum\limits_{i=1}^{n} \left|\lambda_i(A(G_S))-\frac{\sigma}{n}+\frac{\sigma}{n}\right| <  \\
& \sum\limits_{i=1}^{n} \left|\lambda_i(A(G_S))-\frac{\sigma}{n} \right|+\sigma= \mathcal{E}(G_S)+\sigma,
\end{align*}
where the last inequality follows from the fact that $A(G_S)$ has at least one non-positive eigenvalue  \cite[Theorem 2.6]{akbari2023selfloop}.
\end{proof}

In a similar way, one can prove the following statement:

\begin{theorem}
Let $G$ be a graph of order $n$ $(n\geq 2)$, and let $Q$ be a set of vertices which form a clique in $G$, such that $|Q|=\sigma$, $1\leq\sigma\leq n-1$. Let $G_Q$ be a graph obtained by attaching a loop at each vertex in the set $Q$. Then $$\mathcal{E}(G\setminus Q) < \mathcal{E}(G_Q).$$
\end{theorem}

Let $G$ be a graph of order $n$ $(n\geq 2)$, and let $G_S$ be the self-loop graph of $G$ such that $S\subseteq V(G)$ and $|S|=\sigma$. 
Following \cite{Marczyk1985}, we can define:

\begin{definition}
The \emph{line graph} $L(G_S)$ of $G_S$ is a graph whose vertices are the edges of $G_S$, with two vertices in $L(G_S)$ adjacent whenever the corresponding edges in $G_S$ have exactly one vertex in common. Each loop attached at a vertex $v$ in $G_S$ is the vertex with a loop in $L(G_S)$, and this vertex is adjacent to those vertices in $L(G_S)$ which correspond to the edges of $G_S$ incident with the vertex $v$ in $G_S$.
\end{definition}

Let us notice that vertices with loops in $L(G_S)$ form a set of independent vertices, as well as that the line graph $L(G)$ of $G$ is an exact (i.e. induced) subgraph of $L(G_S)$.

From the definition, it follows that the adjacency matrix $A(L(G_S))$ of $L(G_S)$ is of the following form
$$A(L(G_S))=\left(
\begin{array}{cc}
A(L(G)) & M^{T} \\
M & I_{\sigma} \\
\end{array}
\right)
.$$
Here, $M$ stands for the loop-edge adjacency matrix. Precisely, if $G_S$ is a self-loop graph of order $n$, size $m$ and $|S|=\sigma$, then $M=(m_{ij})_{\sigma\times m}$, where $m_{ij}=1$ if the loop $i$ is adjacent to the edge $j$ in $L(G_S)$, and $m_{ij}=0$ otherwise. Observe that when $\sigma=n$, 
then $M=B(G)$, where $B(G)$ is the incidence matrix of $G$.

\begin{corollary}
Let $G$ be a graph of order $n$ $(n\geq 2)$. Suppose $\emptyset\neq S\subseteq V(G)$. 
Then,
$$\mathcal{E}(L(G))<\mathcal{E}(L(G_S)).$$
\end{corollary}

\begin{proof}
The proof follows from Theorem \ref{independent} applied to the set $S$ of vertices with loops in $L(G_S)$. 
\end{proof}

The \emph{incidence matrix} $B(G_S)$ of $G_S$ can be defined in full analogy with the incidence matrix $B(G)$ of $G$. Suppose that $G$ is a graph of order $n$ and size $m$, with the set of vertices $V(G)=\{v_1, v_2,\ldots, v_n\}$, and the set of edges $E(G)=\{e_1, e_2,\ldots, e_m\}$. Let $\mathcal{L}=\{e_{m+1}, e_{m+2},\ldots, e_{m+\sigma}\}$ be the set of loops of $G_S$. The incidence matrix $B(G_S)=(b^S_{ij})$ of $G_S$ is an $n\times (m+\sigma)$ matrix defined as 
$$b^S_{ij}=\left\{
           \begin{array}{ll}
             1, & \hbox{if $v_i$ and $e_j$ are incident;} \\
             0, & \hbox{otherwise.}
           \end{array}
         \right.
$$
Actually, $B(G_S)=\left(B(G)\,|\,N\right)$, where $N$ can be interpreted as the vertex-loop incidence matrix of $G_S$. Precisely, $N$ is a $n\times \sigma$ matrix, such that in each column of $N$ there is exactly one $1$, and all other entries are equal to $0$.

We find it is convenient to define the incidence matrix of a self-loop graph $G_S$ in this way. Namely, if we were to define the \textit{signless Laplacian matrix} $Q(G_S)=(q_{ij})_{n \times n}$ of $G_S$ by analogy with how the Laplacian matrix of a self-loop graph is defined in \cite{anchan20232}, i.e.
$$q_{ij}=\left\{
  \begin{array}{lll}
    1, & \hbox{if vertices $v_i$ and $v_j$ are adjacent,} \\
    0, & \hbox{if vertices $v_i$ and $v_j$ are not adjacent,} \\
    d_G(v_i)+1, & \hbox{if $i=j$ and $v_i\in S$,}\\
    d_G(v_i), & \hbox{if $i=j$ and $v_i\not\in S$,}
  \end{array}
\right.$$
we would obtain that $B(G_S)B(G_S)^T=Q(G_S)$. An analogous relation holds for the corresponding matrices of graphs without loops, i.e., $B(G)B(G)^T=Q(G)$,  where $Q(G)$ is the signless Laplacian matrix of $G,$ see \cite[Equation (7.29)]{cvetkovic2010intro}.

\begin{theorem}
For every eigenvalue $\lambda$ of $L(G_S),$ $\lambda\geq -2$.
\end{theorem}

\begin{proof}
Let $N^T=(\mathfrak{n}_{ij})_{\sigma\times n}$ and $B(G)=(b_{jk})_{n\times m}$. The $(i,k)$-th entry of the matrix $N^TB(G)$ is equal to $\sum_{p=1}^{n} \mathfrak{n}_{ip}b_{pk}$. The addition $\mathfrak{n}_{ip}b_{pk}$, for every $p=1, 2, \ldots, n$, is equal to 1 when $\mathfrak{n}_{ip}=1$ and $b_{pk}=1$, which means that both the loop $i$ and the edge $k$ are incident with the vertex $p$ in $G$, i.e. that the loop $i$ is adjacent to the edge $k$ in $L(G_S)$. Since there is at most one loop attached at each vertex of $G$, it follows that $N^TB(G)=M$. So, we have:
\begin{align*}
    B(G_S)^T B(G_S) & =\left(
  \begin{array}{cc}
    B(G)^TB(G) & B(G)^TN \\
    N^TB(G) & N^TN \\
  \end{array}
\right)= \left(
  \begin{array}{cc}
    A(L(G))+2I_m & M^T \\
    M & I_\sigma \\
  \end{array}
\right)\\
& = A(L(G_S))+2\,\mathfrak{I}_m,
\end{align*}
since $ B(G)^TB(G)=A(L(G))+2I_m$ (see \cite[Equality (1.2)]{cvetkovic2010intro}). Here, $\mathfrak{I}_m$ is the square matrix of order $m+\sigma$ with exactly $m$ ones on the main diagonal, and all other entries equal to zero, i.e.
$$\mathfrak{I}_m=\left(
                  \begin{array}{cc}
                    I_m & O^T \\
                    O & O \\
                  \end{array}
                \right)
.$$
For any vector $x\in \mathbb{R}^{m+\sigma}$, we find
$$x^T B(G_S)^T B(G_S)\,x=(B(G_S) x)^T B(G_S)\,x\geq 0,$$
which means that $B(G_S)^T B(G_S)$ is a positive-se\-mi\-de\-fi\-ni\-te matrix. 
By applying Inequality (\ref{eq2}) from Theorem \ref{Weylthm} to matrices $B(G_S)^T B(G_S)$ and $-2\,\mathfrak{I}_m$, we obtain
\[
\lambda_{m+\sigma}(L(G_S))\geq \lambda_{m+\sigma}(B(G_S)^T B(G_S))+\lambda_{m+\sigma}(-2\,\mathfrak{I}_m)\geq-2. \qedhere
\]
\end{proof}

The following statement will be used in the proof of Theorem \ref{characteristic_line}.

\begin{lemma}\cite[Equation (2.29)]{cvetkovic2010intro} \label{lemma}
Let $M$ be a non-singular square matrix. If $Q$ is also a square matrix, then, 
$$\det\left(
        \begin{array}{cc}
          M & N \\
          P & Q \\
        \end{array}
      \right)
=\det(M) \det(Q-P\,M^{-1}\,N).$$
\end{lemma}

\begin{theorem}\label{characteristic_line}
Let $G$ be a graph of order $n$ and size $m$. Then,
$$P_{L(\widehat{G})}(\lambda)=(\lambda-1)^{n-m}\lambda^m P_{L(G)}\left(\frac{\lambda^2-\lambda-2}{\lambda}\right).$$
\end{theorem}

\begin{proof}
Since the adjacency matrix $A(L(\widehat{G}))$ of $L(\widehat{G})$ is of the form
$$A(L(\widehat{G}))=\left(
              \begin{array}{cc}
                I_n & B(G) \\
                B(G)^T & A(L(G)) \\
              \end{array}
            \right)
,$$
by Lemma \ref{lemma}, we obtain:
\begin{align*}
P_{L(\widehat{G})}(\lambda)= & \det\left(\lambda I_{m+n}-A(L(\widehat{G}))\right)\\
                   = & \det\left((\lambda-1)I_n\right)\det\left(\lambda I_m-A(L(G))-B(G)^T((\lambda-1)I_n)^{-1}B(G)\right)\\
                   = & (\lambda-1)^n\,\det\left(\lambda I_m-A(L(G))-\frac{1}{\lambda-1}B(G)^T\,B(G)\right)\\
                   = & (\lambda-1)^n\,\det\left(\lambda I_m-A(L(G))-\frac{1}{\lambda-1}\left(A(L(G))+2I_m\right)\right)\\
                   = & (\lambda-1)^n\,\det\left(\frac{\lambda^2-\lambda-2}{\lambda-1}I_m-\frac{\lambda}{\lambda-1}A(L(G))\right)\\
                   = & (\lambda-1)^{n-m}\lambda^m\det\left(\frac{\lambda^2-\lambda-2}{\lambda}I_m-A(L(G))\right)\\
                   = & (\lambda-1)^{n-m}\lambda^mP_{L(G)}\left(\frac{\lambda^2-\lambda-2}{\lambda}\right).  \qedhere
\end{align*}
\end{proof}

\section{Some new bounds for the eigenvalues of self-loop graphs and its energy}
\label{Sec3}



In this section, some bounds for the eigenvalues of a self-loop graph are exposed, together with some lower bounds for the energy of such a graph. We start with the following statement:

\begin{theorem}\label{spectrum}
Let $G$ be a graph of order $n$, whose eigenvalues with respect to the adjacency matrix are $\lambda_1(G)\geq\lambda_2(G)\geq\cdots\geq\lambda_n(G)$. Let $G_S$ be the self-loop graph of $G$, such that $S\subseteq V(G)$ and $|S|=\sigma$ $(0\leq\sigma\leq n)$. Then, for the eigenvalues $\lambda_i(G_S)$, $1\leq i\leq n$, of $G_S$ the following holds
\begin{equation}\label{eq:CW}
\lambda_i(G)\leq\lambda_i(G_S)\leq\lambda_i(G)+1, \, 1\leq i\leq n.
\end{equation}
The left-hand side inequality in \eqref{eq:CW} is attained for $\sigma=0$, while the right-hand side inequality for $\sigma=n$.
\end{theorem}

\begin{proof}
The adjacency matrix $A(G_S)$ of $G_S$ is of the form $A(G_S)=A(G)+\mathfrak{I}_{S}$, where $A(G)$ is the adjacency matrix of $G$, while $\mathfrak{I}_{S}$ is a matrix with exactly $\sigma$ ones on the main diagonal and all other entries equal to zero. The eigenvalues of $\mathfrak{I}_{S}$ are $[1]^{\sigma}$ and $[0]^{n-\sigma}$.

\noindent Let us suppose that $0<\sigma<n$. Then, for $i=j$, from Inequality \eqref{eq1}, we find
$$\lambda_i(G_S)\leq\lambda_i(G)+1,\, 1\leq i\leq n,$$
while from Inequality \eqref{eq2}, we obtain
$$\lambda_i(G_S)\geq \lambda_i(G),\, 1\leq i\leq n.$$
For the equalities, we adopt the argument of \cite[Lemma 1, proof]{gutman2022energy}: the case when $\sigma=0$ is clear because $G_S$ coincides with $G;$ when $\sigma=n,$ $\lambda_i(G_S)=\lambda_i(G)+1$ due to $A(G_S)=A(G)+I_n.$  
\end{proof}

\begin{remark}
Since $\lambda_1(G)\leq \Delta$, using Theorem \ref{spectrum}, we obtain $\lambda_1(G_S)\leq \Delta +1$, which is given by  \cite[Theorem 5.1]{akbari2023selfloop}.
\end{remark}


\begin{corollary} \label{Stanleyineq}
Let $G_S$ be the self-loop graph of a graph $G$ of size $m$, and $\lambda_1(G_S)$ be the spectral radius of $G_S.$ Then,
\begin{equation}\label{sradius}
\lambda_1(G_S) \leq \frac{1}{2}(1+\sqrt{1+8m}).
\end{equation}
\end{corollary}

\begin{proof}
By Stanley's Inequality in \cite{Stanley1987}, $\lambda_1(G) \leq \frac12 (-1 + \sqrt{1+8m}).$ The inequality (\ref{sradius}) follows immediately from Theorem~\ref{spectrum}.
\end{proof}


Let us recall the following statement.

\begin{theorem}\cite[Theorem 5.2]{akbari2023selfloop}
\label{thm5.2}
Let $G$ be a connected graph of order $n$ and size $m$. If $S \subseteq V(G)$ with $|S|=\sigma,$ then
\[
\lambda_1(G_S) \geq \frac{2m}{n} + \frac{\sigma}{n}.
\]
If $G$ is a $(k,k+1)$-bidegreed graph for some $k \in \N,$ such that 
\[
 d_G(v) = 
\begin{cases}
k, & \text{ if } v\in S, \\
k+1, & \text{ if } v \in V(G)\backslash S,
\end{cases}
\]
where $d_G$ is the degree of vertices of $G,$ then $\lambda_1(G_S) = \dfrac{2m}{n} + \dfrac{\sigma}{n}.$
\end{theorem}

\begin{theorem}\label{bound}
Let $G$ be a graph of order $n$ $(n\geq 2)$ and $S\subseteq V(G), |S|=\sigma$. Then,
$$\mathcal{E}(G_S)\geq 2\lambda_1(G_S)-\frac{2\sigma}{n}.$$
The equality is attained if $G_S=K_n^{\sigma}$, where $K_n^{\sigma}$ is a graph obtained by attaching a loop at each of $\sigma$ $(0\leq\sigma\leq n)$ chosen vertices of the complete graph $K_n$.
\end{theorem}

\begin{proof}
Let 
$\lambda_1(G_S)\geq\lambda_2(G_S)\geq\cdots\geq\lambda_n(G_S)$ be the eigenvalues of $G_S$, and let $k$ $(1\leq k\leq n)$ be the greatest integer such that $\lambda_k(G_S)\geq \frac{\sigma}{n}$. Since $\sum\limits_{i=1}^{n} \lambda_i(G_S)=\sigma$, we have $\sum\limits_{i=1}^{n} \left(\lambda_i(G_S)-\frac{\sigma}{n}\right)=0$. So,
$$\mathcal{E}(G_S)=\sum\limits_{i=1}^{n} \left|\lambda_i(G_S)-\frac{\sigma}{n}\right|=2\,\sum\limits_{i=1}^{k} \left(\lambda_i(G_S)-\frac{\sigma}{n}\right)\geq 2\,\left(\lambda_1(G_S)-\frac{\sigma}{n}\right),$$ since according to Theorem~\ref{thm5.2}, 
it holds that $\lambda_1(G_S)\geq\frac{\sigma}{n}$.
The second part of the statement can be verified by the direct computation using the results of   Theorem~\ref{KnSchar} in the next section.
\end{proof}


Using Theorems \ref{thm5.2} and \ref{bound},
it follows:

\begin{corollary}\label{corollary1}
Let $G$ be a connected graph of order $n$ $(n\geq 2)$ and size $m$, and let  $S\subseteq V(G)$ and $|S|=\sigma$ $(0\leq\sigma\leq n)$. Then
$$\mathcal{E}(G_S)\geq \frac{4m}{n}.$$
\end{corollary}

Now, let us recall the statements which we will use in the proof of Corollary \ref{corollary2}.

\begin{theorem}\label{p1}
\cite[Theorem 3.2.1]{cvetkovic2010intro}
Let $\lambda_1(G)$ be the index of a graph $G$, and let $\overline{d}$ be its average degree. Then $\overline{d}\leq \lambda_1(G)\leq \Delta$. Moreover, $\overline{d}=\lambda_1(G)$ if and only if $G$ is regular. For a connected graph $G$, $\lambda_1(G)=\Delta$ if and only if $G$ is regular.
\end{theorem}

\begin{theorem}\label{p2}\cite[Theorem 6.7]{cvetkovic1995spectra}
A graph has exactly one positive eigenvalue if and only if its non-isolated vertices form a complete multipartite graph.
\end{theorem}

\begin{corollary}\label{corollary2}
Let $G$ be a regular complete multipartite graph. Then, for every $S \subseteq V(G),$ $\mathcal{E}(G_S)\geq \mathcal{E}(G)$.
\end{corollary}

\begin{proof}
If we suppose that $G$ has $k$ classes of $p$ vertices, i.e. the order $n$ of $G$ is $n=k\,p$, then the spectrum of $G$ is (see \cite{cvetkovic1995spectra}, p.73): 
${\rm Spec}(G)=\left( \begin{array}{ccc}
 n-p & 0 & -p \\
1 & n-k & k-1 \\
\end{array}
 \right).$
So we find $\mathcal{E}(G)=2\lambda_1(G)=2(n-p)$. Since $G$ is an $(n-p)$-regular, $|E(G)|=\frac{n(n-p)}{2}$, which means $\mathcal{E}(G)=\frac{4}{n}|E(G)|$. Therefore, the proof follows from Corollary \ref{corollary1}.
\end{proof}

Next, we provide an analogous result to Theorem~\ref{bound}. Consequently, the regularity condition in Corollary~\ref{corollary2} can be relaxed under certain conditions.

\begin{theorem}
Let $G$ be a graph of order $n \geq 4.$ Let $S\subset V(G)$, where $|S|=\sigma$ and $1\leq\sigma\leq\frac{n}{2}$. If the induced subgraph on $S$ is not a complete graph, then
\[
\mathcal{E}(G_S) \geq 2 \lambda_1(G_S).
\]
\end{theorem}

\begin{proof}
It holds $\sum^n_{i=1} \lambda_i(G_S) = \sigma,$ so 
\begin{equation}
\label{energyt2}
\cE(G_S) = 2 \sum^t_{i=1} \left( \lambda_i(G_S) - \frac{\sigma}{n} \right),
\end{equation}
where $t$ is the greatest integer such that $\lambda_i(G_S) \geq \frac{\sigma}{n}.$ Since the induced subgraph on $S$ is not a complete graph, there are two non-adjacent vertices in $S,$ i.e., $2\widehat{K_1}$ is an induced subgraph of $G_S.$ By Theorem~\ref{interlacing}, we get
\[
\lambda_i (G_S) \geq \lambda_i (2 \widehat{K_1}) = 1, \quad i=1,2.
\]
Thus, it follows from \eqref{energyt2} and $\sigma \leq \frac{n}{2}$ that
\[
\cE(G_S) \geq 2\lambda_1(G_S) + 2\left( 1- \frac{2\sigma}{n}\right) \geq 2 \lambda_1(G_S). \qedhere
\]
\end{proof}

\begin{corollary}
If $G$ is a complete multipartite graph of order $n \geq 4$ and $S \subset V(G),$ $1 \leq |S|\leq \frac{n}{2},$ and the induced subgraph on $S$ is not a complete graph, then $\mathcal{E}(G_S) \geq \mathcal{E}(G).$   
\end{corollary}

\section{
Nordhaus-Gaddum-type Bounds for Spectral Radius of Self-loop Graphs}
\label{Sec4}

In this section, we establish Nordhaus-Gaddum-type bounds, as well as some lower and upper bounds, for the spectral radius of $G_S.$ Precisely, we give a new definition of the complement $\overline{G_S}$ of $G_S$ (see Definition~\ref{complement}) and find bounds for the sum of some particular eigenvalues of $G_S$ and $\overline{G_S}$.

We will start with an upper bound for the spectral radius of $G_S.$

\begin{theorem}
Let $G_S$ be the self-loop graph of $G$ of order $n \geq 2,$ size $m \geq 1,$ and $|S|=\sigma.$ Then,
\begin{equation}
\label{eq:bound1}
\lambda_1(G_S) \leq \frac{\sigma}{n} + \sqrt{\frac{\sigma(n-1)(n-\sigma)}{n^2} + \frac{2m(n-1)}{n}}.
\end{equation}
\end{theorem}

\begin{proof}
In the following, for simplicity, we write $\lambda_i=\lambda_i(G_S), i=1,...,n.$
By Lemma~\ref{lem1}(i),
\begin{equation}
\label{eq:1}
\lambda_1 - \sigma = - (\lambda_2 + \cdots + \lambda_n).
\end{equation}
Thus, by Cauchy-Schwarz Inequality, we have
\begin{equation}
\label{eq:2}
    \lambda_2^2 + \cdots + \lambda_n^2 \geq \frac{(\lambda_1-\sigma)^2}{n-1}.
\end{equation}
Lemma~\ref{lem1} and equation \eqref{eq:2} imply that
\begin{align*}
2m+\sigma &= \lambda_1^2 + (\lambda^2_2+\cdots + \lambda_n^2)  \\
&\geq \lambda_1^2 +\frac{(\lambda_1-\sigma)^2}{n-1} \\
&= \frac{n\lambda_1^2-2\sigma \lambda_1 + \sigma^2}{n-1},
\end{align*}
or equivalently,
\begin{equation}
\label{eq:ineq1}
n\lambda_1^2 - 2 \sigma \lambda_1 + \sigma^2- (n-1)(2m+\sigma) \leq 0.
\end{equation}
The roots of $n\lambda_1^2 - 2 \sigma \lambda_1 + \sigma^2- (n-1)(2m+\sigma)=0$ are as follows:
\begin{align*}
x_i
&= \frac{\sigma}{n} \pm \sqrt{\frac{\sigma^2}{n^2} - \frac{\sigma^2}{n} + \frac{n-1}{n}(2m+\sigma)} \\
&= \frac{\sigma}{n} \pm \sqrt{\frac{\sigma(n-1)(n-\sigma)}{n^2} + \frac{2m(n-1)}{n}}.
\end{align*}
It follows from \eqref{eq:ineq1} that
\[
\lambda_1 \leq \frac{\sigma}{n} + \sqrt{\frac{\sigma(n-1)(n-\sigma)}{n^2} + \frac{2m(n-1)}{n}}. \qedhere
\]
\end{proof}

\begin{corollary}
When $\sigma=0,$ \eqref{eq:bound1} reduces to $\lambda_1(G) \leq \sqrt{\dfrac{2m(n-1)}{n} }$ given by Nosal \cite{nosal1970}.
\end{corollary}

\begin{definition}
\label{complement}
Let $G_S$ be the self-loop graph of $G$ of order $n,$ size $m,$ and $S \subseteq V(G)$ with $|S|=\sigma.$ 
Define the complement $\overline{G_S}$ of $G_S$ to be $\overline{G}_S,$ i.e., the graph obtained from $\overline{G}$ by attaching loops at (its) vertices belonging to the set $S$.
\end{definition}

This definition of $\overline{G_S}$ satisfies the property $\overline{\overline{G_S}} = G_S.$ There is another possible definition of $\overline{G_S}$ recently considered in \cite{ShettyBhat2023}, by taking $\overline{G_S} := \overline{G}_{V \backslash S}.$ Despite the latter definition also satisfies $\overline{\overline{G_S}} = G_S,$ our definition of $\overline{G_S}$ gives a natural generalization of a Nordhaus-Gaddum-type bound for $\lambda_1(G_S)$ and $\lambda_1(\overline{G_S}),$ in the sense that we recover the classical case when $\sigma=0,$ see Remark~\ref{Nosal2}.

\begin{theorem}
\label{thm2} 
Let $G_S$ be a connected self-loop graph of order $n \geq 2,$ size $m \geq 1,$  with $|S|=\sigma,$ such that its complement $\overline{G_S}$ of size $\overline{m}$ is also connected. Let $\lambda_1$ and $\overline{\lambda}_1$ be the spectral radius of $G_S$ and $\overline{G_S}$, respectively. Then,
\begin{equation}
\label{eq:bound2}
n-1+ \frac{2\sigma}{n} \leq \lambda_1+\overline{\lambda}_1
\leq \frac{2\sigma}{n} + \sqrt{2} \sqrt{\frac{2\sigma(n-1)(n-\sigma)}{n^2} + (n-1)^2}.
\end{equation}
\end{theorem}

\begin{proof}
Recall that
\begin{equation}
\label{eq:property1}
m+\overline{m}= \frac{n(n-1)}{2}.
\end{equation}
By Theorem~\ref{thm5.2}, 
we have $\lambda_1 \geq \dfrac{2m}{n}+ \dfrac{\sigma}{n},$ thus
\begin{equation}
\lambda_1 + \overline{\lambda}_1\geq \frac{2(m+\overline{m})}{n}+ \frac{2\sigma}{n}= n-1+ \frac{2\sigma}{n}.
\end{equation}
This gives the lower bound of \eqref{eq:bound2}.

Next, we show the upper bound of \eqref{eq:bound2}.
Let
\begin{align*}
x&= \frac{\sigma(n-1)(n-\sigma)}{n^2} + \frac{2m(n-1)}{n}, \\
y&= \frac{\sigma(n-1)(n-\sigma)}{n^2} + \frac{2\overline{m}(n-1)}{n}.
\end{align*}
Then, by the property \eqref{eq:property1} again,
\begin{equation}
\label{eq:ineq2}
x+y
= \frac{2\sigma(n-1)(n-\sigma)}{n^2} + (n-1)^2.
\end{equation}
Obviously, we have
\begin{equation}
\label{eq:ineq3}
\sqrt{x} + \sqrt{y}
\leq \sqrt{2}\sqrt{x+y}.
\end{equation}
Thus, combining \eqref{eq:bound1},\eqref{eq:ineq2},\eqref{eq:ineq3}, we obtain
\begin{align*}
\lambda_1 + \overline{\lambda}_1
&\leq \frac{2\sigma}{n} + \sqrt{2}\sqrt{x+y} \\
&= \frac{2\sigma}{n} + \sqrt{2} \sqrt{\frac{2\sigma(n-1)(n-\sigma)}{n^2} + (n-1)^2}.
\end{align*}
This completes the proof.
\end{proof}


\begin{remark}\label{Nosal2}
When $\sigma=0,$ \eqref{eq:bound2} reduces to the classical Nordhaus-Gaddum-type bound for the spectral radius of a simple graph $G$ of order $n,$ given by Nosal \cite{nosal1970}
\begin{equation}
\label{eq:bound33}
n-1 \leq \lambda_1 + \overline{\lambda}_1 \leq \sqrt{2}(n-1).
\end{equation}
\end{remark}

Now, we will expose the upper bound for $\lambda_1+\overline{\lambda_1}$ in terms of $\Delta(G)$ and $\delta(G).$

\begin{theorem}
Let $G$ be a graph of order $n$ and $\lambda_1$ and $\overline{\lambda}_1$ be the spectral radius of $G_S$ and $\overline{G_S}$, respectively. Then, we have 
\begin{equation}
\label{eq:bound3}
\lambda_1 + \bar{\lambda}_1 \leq n+1 + (\Delta(G) - \delta(G)).
\end{equation}
\end{theorem}

\begin{proof}
By Theorem~\ref{spectrum}, $\lambda_1 \leq \Delta(G) +1,$ $\bar{\lambda}_1 \leq (n-1 - \delta(G)) +1.$ Thus, we find \eqref{eq:bound3}.
\end{proof}

\begin{remark}
The upper bound of \eqref{eq:bound3} improves the upper bound of \eqref{eq:bound33} for sufficiently large $n$ and small $\kappa:= \Delta(G)-\delta(G).$ More precisely, $n+1+\kappa < \sqrt{2}(n-1)$ when
\[ n > 4\kappa + 9. \]
\end{remark}

The next result asserts a lower bound of the spectral radius of $G_S$ in terms of the first Zagreb index $M_1(G)$ and the minimum degree $\delta(G)$. Before that, let us recall the following theorem.

\begin{theorem}\cite[Theorem 3.2]{Zhou2000spectral}
\label{Zhou}
Let $A=(a_{ij})$ be an $n \times n$ non-negative symmetric matrix with positive row sums $r_1,r_2,\ldots, r_n.$ Then, $\lambda_1(A) \geq \sqrt{\sum^n_{i=1} {r_i}^2/n}$ with equality if and only if $A$ is regular or semiregular.
\end{theorem}

\begin{theorem}
\label{lowerbound}
Let $G_S$ be a connected graph of order $n,$ size $m,$ and $|S|=\sigma$. Let $\lambda_1(G_S)$ be the spectral radius of $G_S$. Let $M_1(G)$ and $\delta(G)$ be the first Zagreb index and minimum degree of $G,$ respectively.
Then,
\begin{equation}
\label{eq:lowerbound2}
\lambda_1(G_S) \geq  \sqrt{\frac{M_1(G)}{n} + \frac{\sigma}{n}(2\delta(G)+1)}.
\end{equation}
In particular, when 
\begin{enumerate}[(i)]
    \item $G_S \cong \widehat{K_n},$  or, 
    \item $G$ is a $(k,k+1)$-semiregular graph such that $d_G(v)=k$ if $v\in S,$ $d_G(v)=k+1$ if $v \in V(G)\backslash S,$
\end{enumerate} 
then, the equality holds.
\end{theorem}

\begin{proof}
By Theorem~\ref{Zhou}, we have $\sqrt{n} \lambda_1(G_S) \geq \sqrt{\sum^n_{i=1} {r_i}^2},$ where
\begin{align}
\sum^n_{i=1} {r_i}^2
&= \sum_{v \in S} (d_{G} (v) + 1)^2  + \sum_{v \in V\backslash S} d^2_{G}(v) \nonumber \\
& =M_1(G) + 2\sum_{v \in S} d_G(v) + \sigma  \nonumber \\
& \geq M_1(G) + \sigma(2\delta(G) + 1). \label{eq:rowsum2}
\end{align}
Thus, we obtain Inequality \eqref{eq:lowerbound2}.  To see that the equality is attainable, we discuss two cases.

Case 1: $G_S \cong \widehat{K_n}.$ By \cite[Theorem 5.1]{akbari2023selfloop} that asserts that 
$\lambda_1(G_S) = n $ if and only if $G_S \cong \widehat{K_n},$ it suffices to check the right side of Inequality \eqref{eq:lowerbound2}:
\[
\sqrt{\frac{M_1(K_n)}{n} + \frac{n(2(n-1)+1)}{n}}
= n.
\]

Case 2:  Suppose $G$ is a $(k,k+1)$-semiregular graph such that $d_G(v)=k$ if $v\in S,$ $d_G(v)=k+1$ if $v \in V(G)\backslash S.$ (Such graphs exist, e.g. $(K_{3,2})_{S}$ with $S=M,$ $|M|=3.$) By Theorem~\ref{thm5.2},  
we have $\lambda_1(G_S) = \frac{2m+\sigma}{n}= k+1.$ Now, observe that for such graphs,
\[
M_1(G) = \sigma k^2 + (n-\sigma)(k+1)^2 = nk^2 + (n-\sigma)(2k+1).
\]
Thus,
\[
\sqrt{\frac{M_1(G)}{n} + \frac{\sigma(2\delta(G)+1)}{n}}
= \sqrt{k^2 + \frac{(n-\sigma)(2k+1)}{n} + \frac{\sigma(2k+1)}{n}}
=k+1. \qedhere
\]
\end{proof}

\begin{remark}
\begin{enumerate}
\item When $\sigma=0,$  Inequality \eqref{eq:lowerbound2} reduces to the classically known bound $\lambda_1(G) \geq \sqrt{M_1(G)/n}$ between the spectral radius and the first Zagreb index of $G.$
\item  Recently, Shetty and Bhat \cite{ShettyBhat2023} have obtained many results regarding $M_1(G_S).$ In particular, \cite[Theorem 7]{ShettyBhat2023} asserts that
\[
\lambda_1(G_S) \geq \sqrt{\frac{M_1(G_S) - 4(m+\sigma)}{n}+1}
\] under the same assumption as in Theorem~\ref{lowerbound}.  One oughts to compare these two inequalities but with some cautions: our Inequality  \eqref{eq:lowerbound2} only requires $M_1(G).$
First, observe that
\[
M_1(G_S)  = \sum_{v \in V} d^2_{G_{S}}(v) = \sum_{v \in S} (d_G(v) +2)^2 + \sum_{v \in V\backslash S} d^2_G(v) = M_1(G) + 4 \sum_{v \in S} d_G(v)+ 4\sigma.
\]
Thus,
\[
\frac{M_1(G_S)-4(m+\sigma)+n}{n}
= \frac{M_1(G)}{n} + \frac{4\sum_{v\in S} d_G(v) -4m +n}{n}.
\]
It suffices to compare $4\sum_{v\in S} d_G(v) -4m +n$ and $\sigma(2\delta +1).$
Observe further that
\[
4\sum_{v\in S} d_G(v) - 4m = 2 \sum_{v \in S} d_G(v) - 2 \sum_{v \in V\backslash S} d_G(v) \leq 2\sigma \Delta + 2\sigma \delta -2n\delta .
\]
Thus,
\begin{align*}
4\sum_{v\in S} d_G(v) -4m +n - \sigma(2\delta +1)
&\leq 2\sigma \Delta + 2\sigma \delta -2n\delta + n - 2\sigma \delta - \sigma \\
&= 2(\sigma \Delta -n\delta) + n-\sigma.
\end{align*}
For $0 \leq \sigma <n,$ as long as
\[
\frac{\sigma}{n} \leq \frac{2\delta(G)-1}{2\Delta(G)-1},
\]  then we have 
\[
\sqrt{\frac{M_1(G_S)-4(m+\sigma)}{n} +1}  \leq \sqrt{\frac{M_1(G)+ \sigma(2\delta +1)}{n}} \leq \lambda_1(G_S),
\]
i.e., our bound is better than \cite[Theorem 7]{ShettyBhat2023}.
\end{enumerate}
\end{remark}

\begin{remark}
\label{edgeless}
If $G_S=\widehat{\overline{K_n}}$ is the edgeless full-loop graph of order $n$, then the equality in \eqref{eq:lowerbound2} holds.
\end{remark}


Next, we establish a relation between the eigenvalues of $G_S$ and $\overline{G_S}.$ Let us recall the following statement.

\begin{theorem}\cite[Theorem 2.2]{akbari2023selfloop}
\label{KnSchar}
Let $(K_n)_S$ be the self-loop complete graph of order $n$ and $|S|=\sigma.$ Then, $\mathrm{Spec}((K_n)_S)$ is determined by the following three cases:
\begin{enumerate}[(i)]
\item If $\sigma=0,$ then $\mathrm{Spec}((K_n)_S) = \begin{pmatrix} n-1 & -1 \\ 1 & n-1 \end{pmatrix}.$
\item If $0<\sigma<n,$ then
\[
\mathrm{Spec}((K_n)_S) =
\begin{pmatrix}
\frac{(n-1)+\sqrt{(n-1)^2+4\sigma}}{2} &0 & -1 &  \frac{(n-1)-\sqrt{(n-1)^2+4\sigma}}{2}  \\
1 & \sigma -1 & n-\sigma -1 &1
\end{pmatrix}.
\]
\item If $\sigma=n,$ then $\mathrm{Spec}((K_n)_S) = \begin{pmatrix} n & 0 \\ 1 & n-1 \end{pmatrix}.$
\end{enumerate}
\end{theorem}

\begin{theorem}\label{NGupperbound}
Let $G_S$ be a graph with self-loops of order $n$ and $|S|=\sigma.$ 
Then, for $j=2,...,n,$
\[
\lambda_j(G_S) + \lambda_{n-j+2}(\overline{G_S}) \leq
\begin{cases}
-1, &\quad \sigma=0, \\
1, &\quad  0<\sigma\leq n.
\end{cases}
\]
\end{theorem}

\begin{proof}
Let $A=A(G_S),$ $B=A(\overline{G_S}),$ and $C=A+B.$ By \eqref{eq2} in Theorem~\ref{Weylthm},  for $i=2$ and $j=2,...,n,$
\[
\lambda_j(A) + \lambda_{n-j+2}(B) \leq \lambda_2(C).
\]
Observe that $C$ can also be expressed as $A' +B'$ where $A'=A((K_n)_S)$ and $B'= \mathfrak{I}_{S}.$
Then, by \eqref{eq1}, for $i=2,$ we have for $j'=1$ or $j'=2,$ 
\[
\lambda_2(C) \leq \lambda_{j'}(A') + \lambda_{3-j'}(B').
\]
By Theorem~\ref{KnSchar}. $A'=A((K_n)_S)$ has 3 distinct spectral cases:

Case 1:  When $\sigma=0,$ $A'=A(K_n)$ and $B'=\textbf{0}_{n\times n}.$ Thus, $\lambda_{2}(C) = \lambda_{2}(A(K_n))= -1.$

Case 2:  When $\sigma =n,$ $A'=A(\widehat{K_n})$ and $B'= I_n.$ Then, $C= J + I_n$ where $J$ is the all-ones matrix. Thus, $\lambda_{2}(C)=1$.

Case 3: Suppose $0<\sigma<n.$ 
If $j'=1,$ then,
\[
\lambda_2(C) 
\leq \lambda_1(A') + \lambda_2(B') 
=  
\begin{cases}
\frac{(n-1)+ \sqrt{(n-1)^2 + 4\sigma}}{2} + 1, & 1<\sigma <n, \\
\frac{(n-1)+ \sqrt{(n-1)^2 + 4\sigma}}{2}, & \sigma=1.
\end{cases}
\]
If $j'=2,$ then, $
\lambda_2(C) \leq \lambda_2(A') + \lambda_1(B') \leq 0 + 1 =1.$
Thus, when $0<\sigma<n,$ we have $\lambda_2(C) \leq 1.$ 
\end{proof}


\begin{theorem}\cite[Theorem 11]{Nikif2006eigen}
\label{Nikifrovthm}
For $i =2,...,n,$ $\lambda_i(G) + \lambda_{n-i+2}(\overline{G}) \geq -1 -2\sqrt{2s(G)},$ where $s(G)= \sum_{v \in V(G)} \left| d_G(v) - \frac{2m}{n}\right|.$
\end{theorem}

Thus, a lower bound for 
$\lambda_j(G_S) + \lambda_{n-j+2}(\overline{G_S})$
follows immediately from Theorems~\ref{spectrum}  and \ref{Nikifrovthm}.

\begin{corollary}\label{NGlowerbound}
Let $G_S$ be a self-loop graph of order $n$ and $|S|=\sigma.$ Let $\overline{G_S}$ be the complement of $G_S.$ Then, for $j=2,...,n,$
$\lambda_j(G_S) + \lambda_{n-j+2}(\overline{G_S}) \geq -1- 2\sqrt{2s(G)}.$
\end{corollary}

\section{Nordhaus-Gaddum-type Bounds for the Energy of a Self-loop Graph}
\label{Sec5}

In this section, we present a Nordhaus-Gaddum-type bound for the energy of a graph with self-loops, in terms of its order $n$ and the number of loops $\sigma.$

First of all, let us recall some statements that we will use in the proof of Theorem \ref{NGEnergy}. 

\begin{theorem}
\cite{DaySo2007}\label{DaySo2007}
Let $X$, $Y$ and $Z$ be square matrices of order $n$, such that $X+Y=Z$. Then
$$\sum_{i=1}^{n} s_i(X)+\sum_{i=1}^{n} s_i(Y)\geq\sum_{i=1}^{n} s_i(Z),$$
where $s_i(M)$, $i=1,2,\ldots, n$, are the singular values of the matrix $M$. Equality holds if and only if there exists an orthogonal matrix $P$, such that $PX$ and $PY$ are both positive semi-definite.
\end{theorem}

\begin{theorem} (Corollary 1.3.13 and Theorem 1.3.14 from \cite{cvetkovic2010intro})\label{quotient}
Let $A$ be a real symmetric matrix whose rows and columns are indexed by $\{1,2,\ldots, n\}$, and with eigenvalues $\lambda_1\geq\lambda_2\geq\cdots\geq\lambda_n$.  Given a partition $\{1, 2,\ldots, n\}=\Delta_1\,\dot{\cup}\,\Delta_2\,\dot{\cup}\cdots\dot{\cup}\,\Delta_m$ with $|\Delta_i|=n_i>0$, consider the corresponding blocking $A=(A_{ij})$, where $A_{ij}$ is an $n_i\times n_j$ block. Let $e_{ij}$ be the sum of the entries in $A_{ij}$ and set $B=(e_{ij}/n_i)$ (note that $e_{ij}/n_i$ is the average row sum in $A_{ij}$). Let us suppose that the block $A_{ij}$ has constant row sums $b_{ij}$, and let $B=(b_{ij})$. Then the spectrum of $B$ is contained in the spectrum of $A$ (taking into account the multiplicities of the eigenvalues).
\end{theorem}

The matrix $B$ from Theorem \ref{quotient} is known as the \textit{quotient matrix} (see, for example \cite{BrouwerHaemers}). Besides, in case the row sum of each block $A_{ij}$ is constant, then the partition is called \textit{equitable} \cite{BrouwerHaemers}.

\begin{theorem}\label{NGEnergy}
Let $G$ be a graph of order $n\geq 2,$ size $m\geq 1$, and $S\subseteq V(G)$, where $|S|=\sigma$, $0\leq\sigma\leq n$. Then,
\[
\mathcal{L}_{n,\sigma}\leq \mathcal{E}(G_S)+\mathcal{E}(\overline{G_S})\leq \mathcal{U}_{n,\sigma},
\]
where
\begin{align*}
\mathcal{L}_{n,\sigma} = & \,(\sigma-1)\left|1-\frac{2\sigma}{n}\right|+(n-\sigma-1)\left(1+\frac{2\sigma}{n}\right)+ \frac{1}{2}\left|n-\frac{4\sigma}{n}+\sqrt{(n-2)^2+8\sigma}\right| +\\
                         & \, \frac{1}{2}\left|n-\frac{4\sigma}{n}-\sqrt{(n-2)^2+8\sigma}\right|,
\end{align*}
and
$$\mathcal{U}_{n,\sigma}=2\sqrt{2}\,\sqrt{2\sigma(n-1)(n-\sigma)+(n^2-n)^2}.$$
\end{theorem}

\begin{proof}
Let $\lambda_1(G_S)\geq\lambda_2(G_S)\geq\cdots\geq\lambda_n(G_S)$ and $\lambda_1(\overline{G_S})\geq\lambda_2(\overline{G_S})\geq\cdots\geq\lambda_n(\overline{G_S})$ be the eigenvalues of $G_S$ and $\overline{G_S}$, respectively. 

For the upper bound $\mathcal{U}_{n,\sigma}$, let $k_1$ and $k_2$, $1\leq k_1, k_2\leq n$, be the greatest integers such that $\lambda_{k_1}(G_S)\geq\frac{\sigma}{n}$ and $\lambda_{k_2}(\overline{G_S})\geq\frac{\sigma}{n}$. Since $\sum_{i=1}^{n} \lambda_i(G_S)=\sigma$, we have $\sum_{i=1}^{n} \left(\lambda_i(G_S)-\frac{\sigma}{n}\right)=0$, and therefore
$$\mathcal{E}(G_S)=2\,\sum_{i=1}^{k_1}\left(\lambda_i(G_S)-\frac{\sigma}{n}\right)\leq 2\,\sum_{i=1}^{k_1}\left(\lambda_1(G_S)-\frac{\sigma}{n}\right)\leq 2n\,\left(\lambda_1(G_S)-\frac{\sigma}{n}\right).$$
Similarly, we find $\mathcal{E}(\overline{G_S})\leq 2n\,(\lambda_1(\overline{G_S})-\frac{\sigma}{n})$, and so
$$\mathcal{E}(G_S)+\mathcal{E}(\overline{G_S})\leq 2n\,(\lambda_1(G_S)+\lambda_1(\overline{G_S}))-4\sigma.$$
By Theorem \ref{thm2}, we get
$$\mathcal{E}(G_S)+\mathcal{E}(\overline{G_S})\leq2\sqrt{2}\,\sqrt{2\sigma(n-1)(n-\sigma)+(n^2-n)^2}.$$

For the lower bound, $\mathcal{L}_{n,\sigma}$, let us denote
\begin{align*}
& X = A(G_S)-\frac{\sigma}{n}I_n = A(G)+\mathfrak{I}_S-\frac{\sigma}{n}I_n,\\
& Y = A(\overline{G_S})-\frac{\sigma}{n}I_n = A(\overline{G})+\mathfrak{I}_S-\frac{\sigma}{n}I_n=J-A(G)+\mathfrak{I}_S-\left(1+\frac{\sigma}{n}\right)I_n,\\
& M = X+Y = J+2\mathfrak{I}_S-\left(1+\frac{2\sigma}{n}\right)I_n.
\end{align*}

Therefore, the matrix $M$ is of the following form:

$$M=\left(
      \begin{array}{cc}
        J+\left(1-\frac{2\sigma}{n}\right)I_{\sigma} & J \\
        J & J+\left(-\frac{2\sigma}{n}-1\right)I_{n-\sigma} \\
      \end{array}
    \right),
$$
i.e.

$$M=\left(
      \begin{array}{cccccc}
        2-\frac{2\sigma}{n} & 1 & 1 & 1 & \cdots & 1 \\
        \vdots & \ddots & \vdots & \vdots & \cdots & 1 \\
        1 & 1 & 2-\frac{2\sigma}{n} & 1 & \cdots & 1 \\
        1 & 1 & 1 & -\frac{2\sigma}{n} & \cdots & 1 \\
        \vdots & \vdots & \vdots & \vdots & \ddots & \vdots \\
        1 & 1 & 1 & 1 & \cdots & -\frac{2\sigma}{n} \\
      \end{array}
    \right).
$$

Since $\mathrm{null}\left(M+\left(1-\left(2-\frac{2\sigma}{n}\right)\right)I_n\right)\geq \sigma-1$, $M$ has eigenvalue $1-\frac{2\sigma}{n}$, with multiplicity at least $\sigma-1$. Similarly, $M$ has eigenvalue $-1-\frac{2\sigma}{n}$, with multiplicity at least $n-\sigma-1$. Let us determine the remaining two eigenvalues, $x_1$ and $x_2$, of the matrix $M$.

The quotient matrix $B$ which corresponds to the matrix $M$ is
\[
B=\left(
  \begin{array}{cc}
    \sigma+1-\frac{2\sigma}{n} & n-\sigma \\
    \sigma & n-\sigma-1-\frac{2\sigma}{n} \\
  \end{array}
\right),
\]
while the characteristic polynomial of the matrix $B$ is
\[
b(x)=x^2+\left(\frac{4\sigma}{n}-n\right)x+\frac{4\sigma^2}{n^2}-4\sigma-1+n.
\]
According to Theorem \ref{quotient}, the roots of the polynomial $b(x)$ are the two remaining eigenvalues of the matrix $M$, i.e. $x_{1,2}=\frac{1}{2}\left(n-\frac{4\sigma}{n}\pm\sqrt{(n-2)^2+8\sigma}\right)$.







Now, we have:
\begin{align*}
\sum_{i=1}^{n} s_i(M)= & \sum_{i=1}^{n}|\lambda_i(M)|=(\sigma-1)\left|1-\frac{2\sigma}{n}\right|+(n-\sigma-1)\left(1+\frac{2\sigma}{n}\right)+\\
                       & \frac{1}{2}\left|n-\frac{4\sigma}{n}+\sqrt{(n-2)^2+8\sigma}\right| + \frac{1}{2}\left|n-\frac{4\sigma}{n}-\sqrt{(n-2)^2+8\sigma}\right|,
\end{align*}
i.e. $\sum_{i=1}^{n} s_i(M)=\mathcal{L}_{n,\sigma}$. By Theorem \ref{DaySo2007}, we obtain
\begin{align*}
\mathcal{L}_{n,\sigma}= & \sum_{i=1}^{n} s_i(M) \leq \sum_{i=1}^{n} s_i(X)+\sum_{i=1}^{n} s_i(Y)=\sum_{i=1}^{n} |\lambda_i(X)|+\sum_{i=1}^{n} |\lambda_i(Y)|\\
= & \sum_{i=1}^{n}|\lambda_i(G_S)-\frac{\sigma}{n}|+\sum_{i=1}^{n}|\lambda_i(\overline{G_S})-\frac{\sigma}{n}|=\mathcal{E}(G_S)+\mathcal{E}(\overline{G_S}).
\end{align*}
The proof is complete.
\end{proof}

\begin{remark}
If $\frac{n}{2}<\sigma\leq n$ and $n\geq 2$, then the following holds:
\begin{itemize}
  \item $1-\frac{2\sigma}{n}<0$,
  \item $n+\sqrt{(n-2)^2+8\sigma}-\frac{4\sigma}{n}>n+\sqrt{(n-2)^2+8\cdot\frac{n}{2}}-\frac{4}{n}\cdot n=n+\sqrt{n^2+4}-4\geq 2+\sqrt{8}-4>0$,
  \item $\frac{4\sigma}{n}+\sqrt{(n-2)^2+8\sigma}-n>\frac{4}{n}\cdot\frac{n}{2}+\sqrt{(n-2)^2+8\cdot\frac{n}{2}}-n=2+\sqrt{n^2+4}-n>2+\sqrt{n^2}-n>0$,
\end{itemize}
so the lower bound $\mathcal{L}_{n,\sigma}$ reduces to
\begin{align*}
\mathcal{L}_{n,\sigma} = &  (\sigma-1)\left(\frac{2\sigma}{n}-1\right)+(n-\sigma-1)\left(1+\frac{2\sigma}{n}\right)+\frac{1}{2}\left(n+\sqrt{(n-2)^2+8\sigma}-\frac{4\sigma}{n}\right)+\\
&\frac{1}{2}\left(\frac{4\sigma}{n}+\sqrt{(n-2)^2+8\sigma}-n\right),
\end{align*}
i.e. $\mathcal{L}_{n,\sigma}=n-\frac{4\sigma}{n}+\sqrt{(n-2)^2+8\sigma}$.
\end{remark}

\begin{remark}
If $0\leq\sigma\leq\frac{n}{2}$ and $n\geq 2$, then the following holds:
\begin{itemize}
  \item $1-\frac{2\sigma}{n}\geq1-\frac{2}{n}\cdot\frac{n}{2}=0$,
  \item $n-\frac{4\sigma}{n}\geq n-\frac{4}{n}\cdot\frac{n}{2}=n-2\geq0$,
\end{itemize}
and therefore the lower bound $\mathcal{L}_{n,\sigma}$ reduces to
\begin{align*}
\mathcal{L}_{n,\sigma} = & (\sigma-1)\left(1-\frac{2\sigma}{n}\right)+(n-\sigma-1)\left(1+\frac{2\sigma}{n}\right)+\frac{1}{2}\left(n-\frac{4\sigma}{n}+\sqrt{(n-2)^2+8\sigma}\right)+\\
& \frac{1}{2}\left|n-\frac{4\sigma}{n}-\sqrt{(n-2)^2+8\sigma}\right|,
\end{align*}
i.e. $\mathcal{L}_{n,\sigma}=\frac{3n}{2}-2+2\sigma\left(1-\frac{1}{n}-\frac{2\sigma}{n}\right)+\frac{1}{2}\sqrt{(n-2)^2+8\sigma}+\frac{1}{2}\left|n-\frac{4\sigma}{n}-\sqrt{(n-2)^2+8\sigma}\right|$,
since the function $f(\sigma)=n-\frac{4\sigma}{n}-\sqrt{(n-2)^2+8\sigma}$ does not have a constant sign on the interval $\left[0,\frac{n}{2}\right]$.
\end{remark}

\textbf{Acknowledgement.}
The third author acknowledges the support from the Ministry of Higher Education Malaysia for Fundamental Research Grant Scheme with Project Code: 
\linebreak FRGS/1/2021/STG06/USM/02/7. The second author thanks the Serbian Ministry of Science, Technological Development and Innovation, for the support through the Mathematical Institute of the Serbian Academy of Sciences and Arts.
We thank the anonymous reviewers for their constructive suggestions that improve the presentation of this paper.

\vspace{0.5cm}
\textbf{Conflicts of interest.} The authors declare no conflict of interest.

\bibliography{bibliography}{}
\bibliographystyle{amsplain}
\end{document}